\documentclass[11pt,twoside]{amsart}
\usepackage[latin1]{inputenc}
\usepackage[T1]{fontenc}
\usepackage{mathtools}
\usepackage{graphicx}
\usepackage{tikz}

\usetikzlibrary{chains}
\usepackage{tcolorbox}
\usepackage{mathabx}
\PassOptionsToPackage{pdfusetitle,pagebackref,colorlinks}{hyperref}
\usepackage{bookmark}
\hypersetup{
  linkcolor={red!70!black},
  citecolor={green!70!black},
  urlcolor={blue!80!black}
}
\usepackage[latin1]{inputenc}
\usepackage{amsmath}
\usepackage{amsthm}
\usepackage{amssymb}

\usepackage[all]{xy}

\usepackage{hyperref}
\newtheorem{thm}{Theorem}[section]

\newtheorem{prop}[thm]{Proposition}

\newtheorem{lem}[thm]{Lemma}
\newtheorem{cor}[thm]{Corollary}

\theoremstyle{definition}
\newtheorem{definition}[thm]{Definition}
\newtheorem{remark}[thm]{Remark}

\numberwithin{equation}{section}

\renewcommand{\to}{\xymatrix@1@=15pt{\ar[r]&}}
\renewcommand{\rightarrow}{\xymatrix@1@=15pt{\ar[r]&}}
\renewcommand{\leftarrow}{\xymatrix@1@=15pt{&\ar[l]}}
\renewcommand{\mapsto}{\xymatrix@1@=15pt{\ar@{|->}[r]&}}
\renewcommand{\twoheadrightarrow}{\xymatrix@1@=18pt{\ar@{->>}[r]&}}
\renewcommand{\hookrightarrow}{\xymatrix@1@=15pt{\ar@{^(->}[r]&}}
\newcommand{\congpf}{\xymatrix@L=0.6ex@1@=15pt{\ar[r]^-\sim&}}
\renewcommand{\cong}{\simeq}

\newcommand{\kt}{\mathcal T}
\newcommand{\ko}{\mathcal O}
\newcommand{\QQ}{\mathbb Q}
\newcommand{\CC}{\mathbb C}
\newcommand{\ZZ}{\mathbb Z}

\usepackage{comment}
\excludecomment{comment}
\begin{document}

\title[]{Characteristic foliations -- A survey}

\author[F.\ Anella, D.\ Huybrechts]{Fabrizio Anella \& Daniel Huybrechts}

\address{
Institut de Math\'ematiques de Jussieu-Paris rive gauche,
4 Place Jussieu, 75005 Paris, France\\
 \& 
Mathematisches Institut, Universit\"at Bonn, Endenicher Allee 60, 53115 Bonn, Germany}
\email{anella@imj-prg.fr, huybrech@math.uni-bonn.de}

\begin{abstract} \noindent
This is a survey article, with essentially complete proofs,  of
a series of recent results concerning the
geometry of the characteristic foliation on smooth divisors in compact hyperk\"ahler manifolds, starting with work by Hwang--Viehweg \cite{HV}, but also covering articles by Amerik--Campana \cite{AC} and Abugaliev \cite{Ab1,Ab2}.

The restriction of the holomorphic symplectic form on a hyperk\"ahler manifold $X$ to a smooth hypersurface $D\subset X$ leads to a regular foliation ${\mathcal F}\subset{\mathcal T}_D$ of rank one, the characteristic foliation.  The picture is complete in dimension four and shows that the behavior of the leaves of ${\mathcal F}$ on $D$  is determined by the Beauville--Bogomolov square $q(D)$ of $D$. In higher dimensions, some of the results
depend on the abundance conjecture for $D$.

 \vspace{-2mm}
\end{abstract}

\maketitle
{\let\thefootnote\relax\footnotetext{This review was prepared in the context of the
seminar organised by the ERC Synergy Grant HyperK, Grant agreement ID 854361.
The talk was delivered on December 10, 2021.}}
\marginpar{}

\section{Main theorem and motivation}
Throughout, $D\subset X$ denotes a smooth connected hypersurface in a compact hyperk\"ahler manifold $X$ of complex dimension $2n$, i.e.\ $X$ is  a simply connected, compact
K\"ahler manifold such that $H^0(X,\Omega_X^2)$ is spanned by a holomorphic symplectic form $\sigma$. Usually $X$ will be in addition assumed to be projective, although one expects all results to hold in general.

The symplectic form $\sigma$ induces a regular foliation of rank one on $D$, i.e.\ a line sub-bundle ${\mathcal F}\subset\kt_D$. We shall denote a generic leaf of the
foliation by $L$ and its Zariski closure by $\bar L$. The space of leaves will be denoted $D/{\mathcal F}$. These notions will all be recalled in Sections \ref{sec:prep} and \ref{sec:Prep2}.

We will discuss the following table. The ultimate goal, only partially met at the moment, would be to establish the equivalence of all assertions in each row. We will throughout assume $n>1$, but see Remark \ref{rem:K3}.

\begin{center}
{ \begin{tabular}[t]{|c|c|c|c|c|c | c| } \hline
\raisebox{-0.12cm}{}&\raisebox{-0.12cm}{(i)}&\raisebox{-0.12cm}{(ii)}&\raisebox{-0.12cm}{(iii)}&\raisebox{-0.12cm}{(iv)}
\\[6pt]\hline
\raisebox{-0.12cm}{}&\raisebox{-0.12cm}{}&\raisebox{-0.12cm}{}&\raisebox{-0.12cm}{}&\raisebox{-0.12cm}{}\\
\raisebox{-0.12cm}{(1)}&\raisebox{-0.12cm}{$\dim \bar L=1$}&\raisebox{-0.12cm}{$L=\bar L\cong {\mathbb P}^1$}&\raisebox{-0.12cm}{$q(D)<0$}&\raisebox{-0.12cm}{$D$ is uniruled}\\
\raisebox{-0.12cm}{}&\raisebox{-0.12cm}{}&\raisebox{-0.12cm}{}&\raisebox{-0.12cm}{}&\raisebox{-0.12cm}{}\\
\raisebox{-0.12cm}{(2)}&\raisebox{-0.12cm}{$\dim\bar L=n$}&\raisebox{-0.12cm}{$\bar L$ Lagr.\ torus}&\raisebox{-0.12cm}{$q(D)=0$}&\raisebox{-0.12cm}{ $D=f^{-1}H\subset X\stackrel{f}{\to} B$ Lagr.\ fibration}\\
\raisebox{-0.12cm}{}&\raisebox{-0.12cm}{}&\raisebox{-0.12cm}{}&\raisebox{-0.12cm}{}&\raisebox{-0.12cm}{}\\
\raisebox{-0.12cm}{(3)}&\raisebox{-0.12cm}{$\dim \bar L=2n-1$}&\raisebox{-0.12cm}{$\bar L=D$}&\raisebox{-0.12cm}{$q(D)>0$}&\raisebox{-0.12cm}{$D$ is of general type}\\
\raisebox{-0.12cm}{}&\raisebox{-0.12cm}{}&\raisebox{-0.12cm}{}&\raisebox{-0.12cm}{}&\raisebox{-0.12cm}{}\\
 \hline
  \end{tabular}}
\end{center}
\bigskip

The table was originally proposed by Campana, cf.\ the work of Amerik--Guseva  cf.\ \cite{AG}.

Essentially, in each row the four conditions are known or at least expected to be equivalent to each other. The assumption on $D$ to be smooth is essential, see Section \ref{sec:exasing}.
The following serves as a guide for what will be discussed in the subsequent sections, where precise references will be provided.\\

\noindent{\bf Case (1): Closed leaves}, cf.\ \cite{AC,Boucksom,HuyKcone}.
$$(1):~~~~~~~~~\xymatrix{{\rm (i)}\ar@{=>}[r]^-{\S\, \ref{sec:1i-iv}}&{\rm (iv)}\ar@{=>}[r]^-{\S\, \ref{sec:1iv-iii}}&{\rm (iii)}\ar@{=>}[r]^-{\S\, \ref{sec:1iii-iv}}&{\rm (iv)}\ar@{=>}[r]^-{\S\, \ref{sec:1iv-i}}&{\rm (i)}}$$
Additionally, we 
observe the easy equivalence
$\xymatrix{{\rm(i)} \ar@{<=>}[r]^-{\S\, \ref{sec:1i--ii}}&{\rm (ii)}.}$\\

\noindent{\bf Case (2): Lagrangian fibrations}, cf.\ \cite{AG,Ab1,HO}.
$$(2):~~~~~~~~~\xymatrix{{\rm (i)}\ar@{=>}[r]^-{\S\, \ref{sec:2i-iii}}&{\rm (iii)}
\ar@{=>}[r]<4pt>^-{\S\, \ref{sec:2iii-iv}}|{~}&{\rm (iv)} 
\ar@{=>}[l]<4pt>^-{\S\, \ref{sec:2iv-iii}}\ar@{=>}[r]^-{\S\, \ref{sec:2iv-i}}&{\rm (i)}}$$

The implication $\xymatrix@C=20pt{{\rm (iii)}\ar@{=>}[r]|{~} &{\rm (iv)}}$ is currently only proved assuming the abundance conjecture for $D$, so the proof is only complete for $n=2$,
cf.\ \cite{AG}. 

We also address $\xymatrix@C=20pt{{\rm (iv)} \ar@{=>}[r]|{~}&{\rm (ii)}\ar@{=>}[r] &{\rm (i)}}$ in Section \ref{sec:2ii-iii} under the assumption that $\mathcal{O}(D)$ is base point free, that is satisfied for instance if $B$ is smooth.\\

\noindent{\bf Case (3): Dense leaves}, cf.\ \cite{Ab2,AC, HV}.
$$(3):~~~~~~~~~\xymatrix{{\rm (i)}\ar@{=>}[r]^-{\S\, \ref{sec:3i-iii}}|{~}&{\rm (iii)}\ar@{=>}[r]^-{\S\, \ref{sec:3iii-iv}}&{\rm (iv)}\ar@{=>}[r]^-{\S\, \ref{sec:3iv-iii}}&{\rm (iii)}\ar@{=>}[r]^-{\S\, \ref{sec:3iii-i}}&{\rm (i)}}$$
The implication $\xymatrix@C=20pt{{\rm (i)} \ar@{=>}[r]|{{~~}}&{\rm (iii)}}$
 is proved assuming that $\xymatrix@C=20pt{{\rm (iii)}\ar@{=>}[r]&{\rm (iv)}}$
 of Case (2) holds.
Note that the equivalence $\xymatrix@C=20pt{{\rm (i)}\ar@{<=>}[r]&{\rm (ii)}}$
 is clear in this case.\\

Additionally, we will also provide a direct argument for $\xymatrix{{\rm  (iv)}\ar@{=>}[r]^-{\S\, \ref{sec:3iv-i}}&{\rm (i)}}$.\\

\begin{remark}\label{rem:K3}
Let us consider the case $n=1$, i.e.\ $X$ a K3 surface. Then, a smooth hypersurface $D\subset X$ is just a smooth curve. Clearly, in this case (i)
holds in all three Cases (1), (2), and (3). The equivalence
of the other conditions (ii), (iii), and (iv)  in (1) and (2) and of (iii) and (iv) in (3)
 is well known.
\end{remark}


\noindent{\bf Acknowledgements.} We would like to thank R.\ Abugaliev and J.-B.\ Bost for discussions, for answering our questions, and for their help with the literature. 
We are particularly grateful to the referee for numerous corrections, instructive comments, and helpful suggestions.

We do not claim any originality for the results presented in this survey, although we sometimes give alternative arguments or provide more details. We hope that the survey contributes to the dissemination
of these results.


\section{Preparations I: Linear algebra of the characteristic foliation}\label{sec:prep}
We collect some linear algebra results and discuss applications to the geometry of the leaves of a foliation.

\subsection{} We begin with discussing some easy linear algebra results
that will be used throughout the later sections.
\smallskip

Let $W$ be a vector space together with a symplectic structure $\sigma$, i.e.\ 
$\sigma\in \bigwedge^2 W^\ast$ such that the induced map $\sigma\colon W\congpf W^\ast$ is an isomorphism. In this situation, the dimension of $W$ is even, so
$\dim W=2n$.

\begin{lem}\label{lem:LA1}
Assume $V\subset W$ is a subspace of codimension one. Then
the subspace
$$F\coloneqq \ker\left(\sigma|_V\colon\!\xymatrix{V\ar@{^(->}[r]& W\ar[r]^-\sigma& W^\ast\ar@{->>}[r]& V^\ast}\!\right)\subset V$$
is of dimension one. 

Similarly, if $U\subset W$ is of codimension two, then
either $\dim\ker(\sigma|_U\colon U\to U^\ast)=2$ or
$\sigma|_U\in \bigwedge^2U^\ast$ is non-degenerate, i.e.\  $\ker(\sigma|_U)=0$.
\end{lem}

\begin{proof} Since $W\congpf W^\ast\twoheadrightarrow V$ has a one-dimensional kernel, we have $\dim F\leq 1$. Furthermore, since $\dim V=2n-1$ is odd, the alternating form $\sigma|_V\in \bigwedge^2V^\ast$ cannot be non-degenerate, i.e.\
$\ker(\sigma|_V)\ne0$. Hence, $\dim F=1$. The proof of the second assertion is analogous.
\end{proof}

\begin{lem}\label{lem:LA2} Assume $V\subset W$ is of codimension one and  let $F=\ker(\sigma|_V)\subset V$. Then $\sigma$ naturally induces a symplectic structure $\bar\sigma$ on $V/F$.
\end{lem}

\begin{proof} By definition of $F$, the restriction $\sigma|_V\colon V\to V^\ast$ factors
 through $V\twoheadrightarrow V/F\,\hookrightarrow V^\ast$. Furthermore, since $\sigma$ is alternating, $\sigma|_V$ takes values in $(V/F)^\ast \subset V^\ast$. Hence, for dimension reasons, $\bar\sigma\colon V/F\congpf (V/F)^\ast$.\end{proof}

Here are a few more concepts from linear algebra: A subspace $U\subset W$ of codimension $c$  of a symplectic vector space $(W,\sigma)$ is called \emph{isotropic} if $\sigma|_U\in \bigwedge^2U^\ast$ is trivial or, equivalently, if $$U\subset U^\perp\coloneqq\ker(W\congpf W^\ast\twoheadrightarrow U^\ast)=\{ w\in W\mid \sigma(U,w)=0\}.$$ If $U^\perp\subset U$, then the subspace $U$ is called \emph{coisotropic}. Since by definition $U^\perp$ is of
dimension $c$, a subspace $U\subset W$ is coisotropic if and only if $\dim \ker(\sigma|_U\colon U\to U^\ast)=c$.
Finally, $U\subset W$ is \emph{Lagrangian} if  $U$ is simultaneously isotropic and coisotropic, i.e.\ $U=U^\perp$.

\begin{lem}\label{lem:Folvert}
Assume $V\subset W$ is of codimension one and let $F=\ker(\sigma|_V)\subset V$.

{\rm (i)}  Then for any Lagrangian subspace $U\subset W$ that is contained in $V$ one has $F\subset U$. 

{\rm (ii)} If a subspace of codimension two $U\subset W$ is contained in $V$ and $F\subset U$, then $U$ is coisotropic.
\end{lem}

\begin{proof} The first claim follows from the commutativity of the diagram
$$\xymatrix@R=20pt{&F\ar@{^(->}[d]&\\
U\ar[d]_0\ar@{^(->}[r]&V\ar[d]_-{\sigma|_V}\ar@{^(->}[r]&W\ar[d]^-\wr_-\sigma\\
U^\ast&V^\ast\ar@{->>}[l]&W^\ast\ar@{->>}[l]}$$
and the assumption that $U$ is Lagrangian, which implies that $U=\ker(W\congpf W^\ast\twoheadrightarrow U^\ast)$.

For the second assertion apply Lemma \ref{lem:LA1}. Since $F\subset U\subset V$, the restriction $\sigma|_U$ is degenerate and, therefore, $\dim \ker(\sigma_U)=2$.
\end{proof}

\subsection{}\label{sec:fol}
A \emph{regular foliation} of a smooth variety (or complex manifold) $D$ is a locally free
subsheaf ${\mathcal F}\subset {\mathcal T}_D$ with a locally free quotient and
such that ${\mathcal F}$ is integrable, i.e.\ $[{\mathcal F},{\mathcal F}]\subset{\mathcal F}$. Note that the integrability condition is automatically satisfied
if ${\rm rk}({\mathcal F})=1$, which is the case of interest to us. 

A \emph{leaf} of a foliation is a maximal connected and immersed complex submanifold
$L\subset D$ (more precisely, a complex manifold together with an injective immersion into $D$) with ${\mathcal F}|_L={\mathcal T}_L$ as subsheaves of
${\mathcal T}_D|_L$.  The integrability condition ensures that there exists a (unique) leaf through any point of $D$.  This is the Frobenius integrability theorem, cf.\ \cite[\S 5, Thm.\ 2]{Manin}. A submanifold $Z\subset D$ is \emph{invariant} under the foliation if ${\mathcal F}|_Z\subset {\mathcal T}_Z$ as subsheaves of ${\mathcal T}_D|_Z$. If $Z$ is a singular subvariety of $D$, then we call $Z$ invariant if its smooth locus is invariant. It is not hard to see that the Zariski closure of an invariant complex submanifold is invariant. Also note that every leaf $L$ intersecting an invariant submanifold $Z\subset D$ is contained in its closure.

A leaf $L\subset D$ is typically not closed. Its Zariski closure $\bar L\subset D$ can be identified with the smallest subvariety containing $L$ that is invariant under the foliation. 
\medskip

Consider now the case of a smooth hypersurface $D\subset X$ of a compact hyperk\"ahler manifold. By virtue of Lemma \ref{lem:LA1}, the kernel 
$${\mathcal F}\coloneqq \ker\left(\sigma|_D\colon\!
\xymatrix{{\mathcal T}_D\ar@{^(->}[r]& {\mathcal T}_X|_D\ar[r]^-\sigma& \Omega_X^\ast|_D\ar@{->>}[r]& \Omega_D^\ast}\!\right)\subset{\mathcal T}_D$$
is a sub-line bundle with locally free kernel. It is called the \emph{characteristic foliation} of the hypersurface $D\subset X$ and was first studied by Hwang and Viehweg \cite{HV}.

\begin{lem}\label{lem:Fomega}
The {normal bundle} of the characteristic foliation ${\mathcal N}_{\mathcal F}\coloneqq {\mathcal T}_D/{\mathcal F}$ is naturally endowed with a symplectic structure and
$${\mathcal F}\cong\omega_D^\ast.$$
In particular, any local transverse section $\Sigma$ of a leaf $L\subset D$ has a natural symplectic structure.
\end{lem}

\begin{proof} The first assertion follows from Lemma \ref{lem:LA2}. The existence of a symplectic structure on ${\mathcal N}_{\mathcal F}$ implies $\det({\mathcal N}_{\mathcal F})\cong{\mathcal O_D}$ and hence
${\mathcal F}\cong \det({\mathcal T}_D)\cong\omega_D^\ast$.
\end{proof}

\begin{remark}
For foliations in general $\det({\mathcal F})^\ast$ is often called the canonical bundle $\omega_{\mathcal F}$ of the foliation. For the characteristic
foliation we thus have $\omega_{\mathcal F}\cong\omega_D$. Note that
for $n=1$, this becomes ${\mathcal F}\cong{\mathcal T}_D$ which is not interesting so that we
usually assume $n>1$. Then, according to \cite[Thm.\ 1.1]{Druel},  $\det({\mathcal F})\cong\omega^\ast_D$ cannot be big and nef.
\end{remark}

The geometric versions of isotropic, coisotropic, and Lagrangian
for subspaces of a symplectic vector space are readily defined:
For example, a subvariety $Z\subset X$ is coisotropic if the rank
of $\sigma|_Z\colon {\mathcal T}_Z\to \Omega_Z$ (say over the smooth
locus of $Z$) is $2\dim(Z)-\dim(X)$ or, equivalently, if ${\rm rk}(\ker(\sigma|_Z))={\rm codim}(Z\subset X)$.

The geometric analogue of Lemma \ref{lem:Folvert} is the following.

\begin{cor}\label{cor:Folvert}
Assume $D\subset X$ is a smooth hypersurface of a compact hyperk\"ahler manifold.

{\rm (i)} If $T\subset X$ is a smooth Lagrangian submanifold that is contained in $D$,  then $T$ is covered by leaves or, equivalently, every leaf $L\subset D$ intersecting $T$ is contained in $T$.

{\rm (ii)} Furthermore, any invariant subvariety $Z\subset X$ of codimension two that is contained in $D\subset X$ is coisotropic.\qed
\end{cor}


\section{Preparations II: Space of leaves}\label{sec:Prep2}
There is no standard text on foliations on complex manifolds or algebraic varieties, but
see e.g.\ \cite{CN}. The arguments typically rely very much on the differentiable theory. The holomorphic version of Reeb's classical stability theorem, cf.\ \cite{HV,Kebekus}, is one example.
We recommend \cite[Sec.\ 2.1]{AC} for further comments.

\subsection{} Consider a foliation ${\mathcal F}$ (of rank one) on a compact complex manifold $D$. The space of leaves is the quotient $D/{\mathcal F}$ by the equivalence
relation that identifies two points if they are contained in the same leaf.
The quotient topology is often complicated and frequently non-Hausdorff,
but the projection $\pi\colon D\to D/{\mathcal F}$ is open, i.e.\ for any open set $U\subset D$ the union of all leaves intersecting $U$ (its saturation) is again an open subset. For more information see \cite[Ch.\ III]{CN}.
A typical example is that of a ${\mathbb P}^1$-bundle $\pi\colon D={\mathbb  P}({\mathcal E})\to Z$ with ${\mathcal F}={\mathcal T}_\pi$. In this situation, 
$D/{\mathcal F}=Z$. We will come back to the local structure of $D/{\mathcal F}$ and the map $\pi\colon D\to D/{\mathcal F}$ in the case that the foliation is  \emph{algebraically integrable}, i.e.\ when every leaf is compact.\footnote{ or, equivalently, when it admits one compact leaf with finite holonomy \cite[Thm.\ 1]{Pereira}, cf.\ Definition \ref{def:hol}.}

\subsection{}\label{sec:Sigma}
 Let $L=\bar L\subset D$ be a compact leaf. For a fixed point $x\in L$
we pick a small transversal section $x\in \Sigma_x\subset D$
(think of it as a germ of a transversal section). Consider a closed
loop $\gamma\colon [0,1]\to L$ with $\gamma(0)=\gamma(1)=x$ and pick
a point $y\in \Sigma_x$ close to $x$. Then there exists a differentiable
map $\Phi\colon \Sigma_x\times[0,1]\to X$ such that $\Phi(y,0)=\Phi(y,1)$, $\Phi(0,t)=\gamma(t)$, and
$\Phi(y,0)=y$. The pull-back of the foliation ${\mathcal F}$ defines a real foliation of rank one on $\Sigma_x\times S^1$. 

Starting with a point $(y,0)$ and integrating 
defines a path $\gamma_y\colon [0,1]\to \Sigma_x\times S^1$ satisfying $\gamma_y(t)=(\rho_{\gamma,y}(t),t)$ and $\gamma_y(0)=(y,0)$. Note, however, that this path is not  necessarily closed, so possibly $\gamma_y(1)\ne (y,0)$. 

It turns out that the map
$y\mapsto \rho_{\gamma,y}(1)$
only depends on the homotopy class of $\gamma$, which gives rise to the following.

\begin{definition}\label{def:hol}
The \emph{holonomy} of a compact leaf $L\subset D$ is the group homomorphism
$$\rho\colon\pi_1(L,x)\to {\rm Diff}(\Sigma_x),~\gamma\mapsto (\rho_\gamma\colon y\mapsto \rho_{\gamma,y}(1)).$$
The leaf has \emph{finite holonomy} if the image of $\rho$,  the \emph{holonomy group}  $$G_L\coloneqq {\rm Im}(\rho)\subset{\rm Diff}(\Sigma_x),$$ is finite.
\end{definition}

Note that the image of $\rho$ only depends on $x$ up to conjugation. In particular,
the property of a leaf to have finite holonomy does not depend on the chosen base point.

Since we are interested in foliations of rank one, a compact leaf will be a
compact complex curve. If this curve is rational, i.e.\ $L=\bar L\cong{\mathbb P}^1$, then it has automatically finite (and in fact trivial) holonomy. Also note 
that due to a result of Holmann \cite[Prop.\ 4.2]{Holmann}, one knows that
if all leaves of a foliation on a K\"ahler manifold are compact,
then they all have finite holonomy.\footnote{In fact, Holman \cite[Prop.\ 4.2]{Holmann} proved that for a holomorphic foliation with only compact leaves stability is equivalent to finite holonomy. Earlier results in this direction are due to Epstein \cite{Epstein}.}

\begin{thm}\label{thm:Reeb}
Assume that a foliation ${\mathcal F}$ of rank one on a smooth projective variety $D$ has one leaf isomorphic to ${\mathbb P}^1$. Then ${\mathcal F}$ is algebraically integrable and all leaves are curves isomorphic to ${\mathbb P}^1$.
\end{thm}

\begin{proof}
According to a result of Pereira \cite[Thm.\ 1]{Pereira}, for a foliation on a compact K\"ahler manifold  the existence of one compact leaf with finite holonomy  implies that all leaves are compact with finite holonomy.  This proves that ${\mathcal F}$ is algebraically integrable. Reeb stability \cite[Prop.\ 2.5]{HV} then implies that all leaves are isomorphic to each other.\footnote{A word on the name `Reeb stability': A leaf $L$ is stable if every open neighbourhood $L\subset U$ of it contains an invariant open neighbourhood $L\subset U'\subset U$. The foliation is stable if all leaves are stable. Reeb stability in the holomorphic context as in \cite{HV,Kebekus} can be viewed as saying that compact leaves with finite holonomy are stable.}
\end{proof}

\subsection{} Let us come back to the space of leaves $D/{\mathcal F}$ 
and the map $\pi\colon D\to D/{\mathcal F}$
for an algebraically integrable foliation (of rank one) on a smooth projective variety $D$. We collect the facts
that will be used at various places later:\footnote{These results seem well
known to the experts but we could not find a source with complete proofs.
Thanks to J.-B.\ Bost for an instructive email exchange.}
\smallskip

$\bullet$ The map $\pi\colon x\mapsto |G_{L_x}|\cdot [L_x]$ defines a holomorphic
map from $D$ to the Chow variety (or Barlet space), cf.\ 
\cite[Prop.\ 2.5]{HV} of \cite{MM03}. Here, $L_x$ is the unique leaf through $x$ and $G_{L_x}$ is its holonomy group. \smallskip 

$\bullet$ The quotient $D/{\mathcal F}$ can be identified with (the normalization of) the image of $\pi$. In particular, $D/{\mathcal F}$ is an algebraic variety
and $\pi\colon D\to D/{\mathcal F}$ is a proper morphism. The map is in general
not flat. Indeed, by `miracle flatness', the flatness of the equidimensional morphism $\pi$ is equivalent to the smoothness of the leaf space $D/{\mathcal F}$, which
is discussed next.
\smallskip

$\bullet$ Assume $x\in \Sigma_x\subset D$ is a transversal section of a leaf $x\in L$ as in Section \ref{sec:Sigma}. Then locally $\Sigma_x/G_x$ is a chart of $D/{\mathcal F}$ at the point corresponding to the leaf $L_x$, cf.\ \cite[Thm.\ 2.4]{HV}. In particular,
$D/{\mathcal F}$ has quotient singularities.\smallskip

$\bullet$ Assume the fibres of $D\to D/{\mathcal F}$ are rational curves. 
 Then the description of local charts shows
that $D/{\mathcal F}$ is a smooth projective variety, for in this case
$\pi_1(L_x)=\{1\}$ and hence $G_{L_x}=\{1\}$. In the differentiable setting this is \cite{He60} and for a discussion in our setting see e.g.\ \cite[Lem.\ 5]{Sawon}.
If, furthermore, ${\mathcal F}$ is the characteristic foliation of a smooth hypersurface
in a projective hyperk\"ahler manifold, then $D/{\mathcal F}$ comes with a natural symplectic structure.\smallskip

$\bullet$ For a dense open subset of points $x\in D$ the leaf $L_x$ through $x$ has trivial holonomy $|G_{L_x}|=1$, cf.\ \cite{EMT77}.\smallskip

$\bullet$ The scheme-theoretic fibre of $\pi\colon D\to D/{\mathcal F}$ over
a point $[L]\in D/{\mathcal F}$ corresponding to a leaf with trivial holonomy
 $|G_L|=1$ is the leaf $L$. The fibre is non-reduced over points with non-trivial holonomy; more precisely, it is a multiple fibre with multiplicity $|G_L|\ne 1$.





\subsection{}
It is easy to prove that a smooth curve $C\subset S$ in a K3 surface with $(C.C)\geq 0$ is nef. The following is the hyperk\"ahler analogue of
this fact.\footnote{We wish to thank R.\ Abugaliev for communicating this result to us. It seems known to some experts, but has not been written down anywhere.}

\begin{prop}\label{prop:Abtrick}
Let $D\subset X$ be a smooth hypersurface in a projective hyperk\"ahler manifold $X$. If
$q(D)\geq 0$, then $D$ is nef.
\end{prop}

\begin{proof}
Assume $D$ is not nef. Then there exists an irreducible
curve $C\subset X$ with $D\cdot C<0$. The latter implies $C\subset D$ and
$\deg(\omega_D|_C)<0$, which by Lemma \ref{lem:Fomega} shows $\deg({\mathcal F}|_C)>0$, i.e.\ ${\mathcal F}|_C$ is ample.

However, the latter implies that the foliation ${\mathcal F}$ is algebraic, i.e.\ all leaves are compact or, equivalently, algebraic curves. This is a consequence of a much more general result that was original proved by Bogomolov--McQuillan \cite{BMcQ} and Bost \cite{Bost} with details provided by Kebekus, Sol\'a Conde, and Toma \cite[Thm.\ 1 \& 2]{Kebekus}: If the restriction of a foliation
to some complete curve $C$ is an ample vector bundle, then the leaf through any point of $C$ is algebraic. Moreover, the leaf through a general point of $C$ is rationally connected and, in fact, all leaves are rationally connected. In fact, according to Theorem \ref{thm:Reeb}, all we need is one compact rational leaf.\footnote{Note that in this sense Reeb stability shows that the ampleness along $C$ determines the behavior of the foliation globally.}

In our situation the result means that all leaves of ${\mathcal F}$ are smooth rational curves and, in particular, $D$ is uniruled. Then, the discussion in Section \ref{sec:1iv-iii}, which is independent of this proposition, leads to the contradiction $q(D)<0$.
\end{proof}

%
\section{Case (1): Closed leaves}\label{sec:Case1}

We are proving the equivalence of the conditions (i)-(iv) in Case (1). All results are unconditional. The main reference for this section is \cite{AC} with priori work 
\cite{Boucksom,HuyKcone}.

\subsection{}\label{sec:1i-iv} (i) $\Rightarrow$ (iv):  We assume that the leaves
of the foliation have one-dimensional closures and want to show that $D$ is then uniruled.\footnote{This part is the most technical one of all of this survey and we will have to be sketchy at points.}

First of all, since the boundary $\bar L\setminus L$ is invariant, it is a union of leaves. However, under our assumption all leaves are one-dimensional and,
therefore, all leaves are in fact closed $\bar L=L$. 
Thus, all leaves are algebraic curves, i.e.\ ${\mathcal F}$ is algebraically integrable, and the natural projection
$$\pi\colon D\to D/{\mathcal F}$$
is a proper  morphism between algebraic varieties.

		
\smallskip

The proof proceeds in three steps.

\begin{description}
\item[{\bf Step 1}] Prove that $\pi$ has no multiple fibres in codimension one and that the canonical divisor of $D/\mathcal{F}$ is trivial.
\item[{\bf Step 2}] Deduce the isotriviality of $\pi$, combining results of Miyaoka and Hwang--Viehweg, and consider a finite quasi-\'etale cover of $D$ that splits into a product.
\item[{\bf Step 3}] Reach a contradiction by considering the numerical dimension of $\omega_D$.
\end{description}

\medskip

\noindent
{\bf Step 1.} Intuitively, the morphism $\pi\colon D\to D/{\mathcal F}$ induced by the algebraically integrable foliation $\mathcal{F}$ contracts all
curves with tangent space contained in the kernel of $\sigma|_D$. Therefore $\sigma|_D$ should descend to a non-degenerate two-form on the a priori singular
space $D/{\mathcal F}$ and so we expect (the smooth part of) $D/{\mathcal F}$ to have trivial canonical bundle.

For the open set covered by leaves with trivial holonomy this can be made rigorous 
by taking a locally transverse section $\Sigma$ to each leaf of the foliation, which can be taken as a local model for $D/{\mathcal F}$ near the point corresponding to the leaf.
 Then, by Lemma \ref{lem:Fomega}, $\sigma|_\Sigma$ is symplectic. These symplectic forms glue to a global symplectic form on the open subset of $D/{\mathcal F}$ corresponding to leaves with trivial holonomy, see \cite[Lem.\ 6]{Sawon} for some more details.
	
	Looking at the local behavior of the symplectic form around the multiple fibres, Amerik and Campana \cite{AC} proved the following:
	
\begin{lem}[Amerik--Campana]
The map $\pi$ has no multiple fibres in codimension one. Moreover some multiple of the canonical bundle of $D/\mathcal{F}$ is trivial.
\end{lem}

\begin{proof}
{ Suppose by contradiction that there exists a divisor $V \subset D/\mathcal{F} $ such that the fibres over $V$ are multiple of order $m>1$. The statement is local around a generic point $0\in V$. We can take a local multisection $W$ over 0 that meets transversally the non-reduced fibres. We choose coordinates $(z, t)$ around $0\in V$ such that $z$ are coordinates for $V$ and $t$ parametrizes the normal direction. 
Thus, we can choose coordinates $(u,s,w)$ around $W$ in such a way that $W$ is given by the equation $w=0$ and $\pi(u,s,w)=(u, s^m)$.
	
By the discussion above, $\sigma^{n-1}=\pi^*\alpha$ for some form of top degree on the base at least over the complement of $V$.
Locally, $\alpha=G(z,t) \cdot dz \wedge dt$, where $dz$ is an $n-2$ form,
and then $|G(z,t)|=e^{g(z, t)}\cdot |t|^{-c}$ for some real-valued bounded function 
$g$. We claim that $c=1-1/m$. 
Assuming for the moment that this is true, then the meromorphic function
$G(z,t)$ has poles of order strictly less than one around $t=0$, which is absurd. 
	
To prove the claim we denote by $\pi_0$ the restriction of $\pi\colon D \rightarrow D/\mathcal{F}$ to $W$. In coordinates $\pi_0(u,s,w)= (u, s^m)$. By the base change formula, we see that the restriction of $\sigma^{n-1}$ to $W$ is
$$\sigma^{n-1}|_W= \pi_0^* \alpha = G(u,s^m) \cdot m \cdot s^{m-1} \cdot du\wedge ds=h( u, s) \cdot du\wedge ds$$ for some function $h(u,s)$ that does not vanish when $s=0$.
Thus, we can write 
$$|G(z, t)|=|G( u, s^m)|=\frac{|h(u, s)|}{m} |s|^{1-m} =e^{g(u, s)} |s|^{1-m}=e^{g(z, t)}|t|^{-1+ 1/m}  $$
which proves the claim.}
\end{proof}

The singular fibres of $\pi\colon D\to D/{\mathcal F}$ are simply multiples of smooth curves. By the above lemma we can assume $\pi$ is smooth over the complement $D^o/{\mathcal F}\subset D/{\mathcal F}$ of a closed set
of codimension two and we may assume that $D^o/{\mathcal F}$ is smooth. Denote by $\pi^o\colon D^o \rightarrow D^o/\mathcal{F}$ the restriction of $\pi$.
	If one leaf is rational, then by Reeb stability, see Theorem \ref{thm:Reeb}, all the leaves are rational curves and we are done.\footnote{One can avoid using Reeb stability here: Instead of showing (i) $\Rightarrow$ (iv) one shows (i) $\Rightarrow$ (iii) and then uses Section \ref{sec:1iii-iv} to complete by (iii) $\Rightarrow$ (iv). Indeed, $q(D)<0$ is equivalent to $\int_D{\rm c}_1({\mathcal F})H^{2n-2}>0$ for some ample divisor $H$ on $D$. The latter follows if one can show $\int_D{\rm c}_1({\mathcal F})\pi^\ast H_0^{2n-2}>0$ for some ample 
divisor $H_0$ on $D/{\mathcal F}$, which in turn would follow
from $\int_L{\rm c}_1({\mathcal F})|_L>0$, i.e.\ $g(L)=0$.}
 So we can assume all the leaves have positive genus and singular, i.e.\ multiple, fibres appear in codimension at least two. 
\medskip

\noindent
{\bf Step 2.} We want to prove that $\pi^o\colon D^o\to D^o/{\mathcal F}$ is isotrivial.
	There are two possibilities depending on the genus $g$ of the general leaf: If $g=1$, then the fibration has to be isotrivial, for otherwise one of the fibres of $\pi$ would be rational, in which case we are done already, cf.\ Theorem \ref{thm:Reeb}.

The isotriviality is less trivial  for $g>1$. It follows from the observation that the following results of Miyaoka--Mori and Hwang--Viehweg contradict each other.
\smallskip

$\bullet$
For any coherent subsheaf $\mathcal{H} \subset \Omega_{D^o/\mathcal{F}} $ one has $\kappa(D^o/\mathcal{F}, \det(\mathcal{H}))\leq 0$. Indeed, according to 
\cite[Cor.\ 8.6]{Miyaoka} or \cite[Thm.\ 1]{MM}, the restriction of
$\Omega_{D^o/{\mathcal F}}$  to a generic complete intersection curve is semi-positive. At the same time, its determinant is trivial. Hence, all sub-sheaves 
of $\Omega_{D^o/{\mathcal F}}$ have
non-positive degree, which leads to the assertion.

\smallskip

$\bullet$ Assuming $g>1$, there exists a coherent subsheaf $\mathcal{H} \subset \Omega_{D^o/\mathcal{F}}$ such that ${\rm var}(\pi^o) =\kappa(D^o/\mathcal{F},\mathcal{H})$, cf.\ \cite[Thm.\ 3.2 \& Prop.\ 4.4]{HV}. Roughly, the relative cotangent sheaf of the natural map $D^o\to {\mathcal M}_g$ provides this sub-sheaf. (Strictly speaking, this is only true after passing to a finite cover of $D^o$ which does not affect the argument.)

\medskip

Once isotriviality for $g=1$ and $g>1$ has been established, one can use the
fact that the moduli space ${\mathcal M}_g(N)$ of curves with a level $N$-structures, $N\geq 3$, is fine to show that there exists a finite \'etale cover $\Delta\to D^o/{\mathcal F}$ such that pull-back $\tilde D\coloneqq D^o\times_{D^o/{\mathcal F}}\Delta$ splits
 as $\tilde{D}\cong  \Delta \times C$, where $C$ is the generic fibre of $\pi$.  Indeed, there exists a finite \'etale cover $\Delta\to
 D^o/\mathcal{F}$ such that the finite cohomology groups
  $H^1(\tilde D_t,\ZZ/N\ZZ)$, $t\in \Delta$, of the fibres of the pull-back family 
  $\tilde\pi^o\colon\tilde D\coloneqq D^o\times_{D^o/{\mathcal F}}\Delta\to \Delta$
 form a trivial local system. The induced morphism $\Delta \to {\mathcal M}_g(N)$
 has the property that the pull-back of the universal curve over ${\mathcal M}_g(N)$ is isomorphic to  $\tilde\pi^o$. However, the isotriviality implies that $\Delta\to {\mathcal M}_g(N)$ is  constant and, therefore, $\tilde D$ splits as claimed. 
 \medskip

\noindent
{\bf Step 3.} 
Since $D^o/{\mathcal F}$ has trivial canonical bundle, the same holds
for $\Delta$. Hence, 
$$\nu (\tilde D,\omega_D)=\kappa(\tilde D)=\begin{cases}0& \text{ if }g=1\\
1& \text{ if }g>1.\end{cases}$$
As the numerical (and also the Kodaira) dimension is preserved under \'etale maps, one finds
$$\nu(D,\omega_D)=\begin{cases} 0&\text{  if }g=1\\
1&\text{ if }g>1.\end{cases}$$

Since $\omega_D=\mathcal{O}_X(D)|_D$, we have $\nu(D,\omega_D)=\nu(X, D)-1$. 
However,  the numerical dimension of a nef divisor in a hyperk\"ahler manifold can be $0, n$ or $2n$. Since $n>1$, the only possibility is that $n=2$ and $g>1$, which is excluded as follows: A fibre $S$ of the canonical map is equivalent as a cycle, up to a multiple, to $D\cdot D$. This means that $S$ is Lagrangian, for $\int_S \sigma \bar{\sigma}= q(D)=0$. Hence,
by Corollary \ref{cor:Folvert},  the leaves of the characteristic foliation must be contained in $S$ and induce a fibration on $S$ of curves of genus at least two. This contradicts the fact that the canonical bundle of $S$ is trivial.

\subsection{}\label{sec:1iv-iii} (iv) $\Rightarrow$ (iii): 
We assume that $D$ is uniruled and will show $q(D)<0$ by excluding
$q(D)>0$ and $q(D)=0$.

Suppose $q(D)>0$. Then $D$ is contained in the interior of the positive cone
and, therefore, also in the interior of the pseudo-effective cone. Hence, $D$ is big
\cite[Lem.\ 2.2.3]{Laz}, i.e.\ $h^0(X,{\mathcal O}(kD))\sim k^{2n}$, which implies $h^0(D,\omega_D^k)\sim k^{2n-1}$ contradicting the assumption that $D$ is uniruled.

Suppose $q(D)=0$. If $D$ is nef, then $\omega_D$ is nef too, which again would contradict the assumption that $D$ uniruled. If $D$ is not contained in the closure of the movable cone, then, by Boucksom's duality of movable and pseudo-effective cone \cite{Boucksom}, it is contained in the interior of the pseudo-effective cone and one argues as above. If $D$ is contained in the boundary of the movable cone, $D$ is the limit of movable divisors and hence its restriction to $D$
is still a limit of effective divisors. However, this implies that 
$\omega_D\cong {\mathcal O}(D)|_D$ is pseudo-effective which contradicts $D$
uniruled.\footnote{We wish to thank R.\ Abugaliev for his help with this argument.}
\smallskip

The discussion should be compared to the result \cite[Thm.\ 3.7]{Loray} asserting
in the general setting that $D$ is uniruled if and only if $\omega_{\mathcal F}$ is not pseudo-effective. The above discussion can be interpreted as saying that any smooth divisor $D\subset X$ with $q(D)\geq 0$ has a pseudo-effective $\omega_D$, thus $D$ cannot be uniruled.

\subsection{}\label{sec:1iii-iv} (iii) $\Rightarrow$ (iv): We assume $q(D)<0$ and want to show that $D$ is then uniruled.  (The smoothness of $D$ is not essential.)
We offer two proofs.
\smallskip

First, it is known that prime exceptional divisors are uniruled \cite[Prop.\ 5.4]{HuyKcone} or \cite[Prop.\ 4.7 \& Thm.\ 4.3]{Boucksom}. Indeed, since the positive cone is self-dual, it contains a class $\alpha$ such that $q(\alpha, D)<0$.
	Hence, there exists a bimeromorphic map between hyperk\"ahler manifolds $f\colon X\dashrightarrow X'$ such that $f_* \alpha =\omega'+\sum a_i D_i'$ for some  uniruled divisors $D_i'$, positive real numbers $a_i$, and a K\"ahler class $\omega'$, cf.\  \cite{HuyKcone} or \cite[Thm.\ 4.3 (ii)]{Boucksom}.
	Since the quadratic form is preserved by $f$, we have $0>q(\alpha, D)=q(\omega'+\sum a_i D_i',f_*D)>\sum a_i q( D_i',f_*D)$ and hence for some $i$ we have $q(D_i', f_* D)<0$. This implies that $f_*D$ and $D_i $ coincide and that in particular $D$ is uniruled since its push-forward in $X'$ is. 
\smallskip

Here is a more direct proof relying on the criterion for uniruledness by Miyaoka
and Mori \cite{MM, Miyaoka}: A smooth projective variety $Z$ of dimension $d$ is uniruled if $\int_Z{\rm c}_1(\omega_Z)\cdot H_Z^{d-1}<0$ for an ample divisor $H_Z$ on $Z$.
Applied to $Z=D$ and observing that $q(D)<0$ implies 
$\int_D{\rm c}_1(D)\cdot H|_D^{2n-2}=\int_X[D]^2\cdot H^{2n-2}<0$
for any ample divisor $H$ on $X$, it implies the result.

\subsection{}\label{sec:1iv-i} (iv) $\Rightarrow$ (i): We assume that
$D$ is uniruled  and want to prove that the leaves are closed.

By assumption, there exists a dominant map $\xymatrix@C=15pt{\varphi\colon{\mathbb P}^1\times V\ar@{..>}[r]& D}$ with $\dim (V)=2n-2$. Since ${\mathbb P}^1$ admits no non-trivial forms of degree one or two, the pull-back of $\sigma$ to ${\mathbb P}^1\times V$ is the pull-back of a holomorphic form on $V$. This readily shows that all $\varphi_t({\mathbb P}^1)\subset D$ are invariant with respect to the foliation. Hence, the generic leaf is of this form, which proves the claim. See \cite[Prop.\ 4.5]{Druel2} for a generalization to singular uniruled divisors.


\subsection{}\label{sec:1i--ii} (i) $\Leftrightarrow$ (ii): Clearly, (ii) implies (i).
 The converse is part of the discussion in Section \ref{sec:1i-iv}.

\section{Case (2): Lagrangian fibrations}\label{sec:Case2}
The equivalence of the conditions (i)-(iv) in Case (2) 
is shown in dimension four. In higher dimensions, the proof assumes
that abundance holds for $D$. The direction (iv) $\Rightarrow$ (i)  is due to Abugaliev and
is the main result of \cite{Ab1}.

\subsection{}\label{sec:2i-iii} (i) $\Rightarrow$ (iii):  We assume $\dim\bar L=n$ and want to exclude $q(D)<0$ and $q(D)>0$.

According to Section \ref{sec:Case1}, $q(D)<0$ implies  $\dim\bar L=1$. To exclude $q(D)>0$, use that according to Section \ref{sec:3iii-i}\footnote{We leave it to the reader to check that the argument is not circular.}
it would imply that the leaves are dense. 

\subsection{}\label{sec:2iii-iv} (iii) $\Rightarrow$ (iv): We assume $q(D)=0$. Then, by Proposition \ref{prop:Abtrick},  $D$ and hence $\omega_D\cong{\mathcal O}(D)|_D$ are nef. Assuming the abundance conjecture for $D$, we know that $\omega_D$ is semi-ample. Hence, 
by \cite[Cor.\ 1.8]{Dem} also $D$\label{foot:DHP} is semi-ample,\footnote{This is a highly non-trivial statement asserting that $H^0(X,{\mathcal O}(kD))\to H^0(D,{\mathcal O}(kD)|_D)$ is surjective for sufficiently divisible $k$. For an alternative, algebraic argument see \cite[Cor.\ 5.2]{AC}.} i.e.\ some power ${\mathcal O}(kD)$ defines a Lagrangian fibration $f\colon X\to B$
and, therefore, $kD$ is the pull-back of a divisor in $B$. Hence, $D$ is vertical.

\begin{remark}
The implication (i) $\Rightarrow$ (iv) in dimension four was first proved by
Amerik and Guseva \cite{AG}.
\end{remark}

\subsection{}\label{sec:2iv-iii} (iv) $\Rightarrow$ (iii): We assume now that
there exists a Lagrangian fibration $X\to B$ such that $D$ is the pre-image (as a set) of a hypersurface $H\subset B$. Then $[D]$ and $f^\ast[H]$ are proportional. Therefore, 
since $[H]^{n+1}=0$, also $[D]^{n+1}=0$ and hence $q(D)=0$.

%
%

\subsection{}\label{sec:2iv-i} (iv) $\Rightarrow$ (i):  We assume that $X$ comes with a Lagrangian fibration $f\colon X\to B$ 
such that $D=f^{-1}(H)$ (as sets) for some hypersurface $H\subset B$ and want to show that the closure of the generic leaf is of dimension $n$ (and in fact a torus). \smallskip

Assume first that $H$ is contained in the discriminant locus $\Delta\subset B$. Then $D$ is algebraically integrable by a result of Hwang and Oguiso \cite[Thm.\ 1.2]{HO}.
By the results of Section \ref{sec:Case1}, the latter implies that $D$ is uniruled and hence $q(D)<0$, which contradicts (iii) that we proved already in Section \ref{sec:2iv-iii}.\footnote{The argument shows that for any component of the discriminant divisor $H\subset B$ the reduction of $f^{-1}(H)$ cannot be smooth. Either it consists of more than one component, with possibly each component
individually smooth, or it is irreducible but singular.}\smallskip

Assume now that $H$ is not contained in the discriminant divisor. Then, since $D$ is smooth, the generic fibre of $f|_D\colon D\to H=f(D)$ is a smooth Lagrangian torus.
By Corollary \ref{cor:Folvert}, the generic leaf is contained in a fibre of $D\to f(D)$.
We have to show that it is dense in there. Note that for $n=2$ the result is immediate. Indeed, if the generic leaf is not dense in the fibre, then by Section \ref{sec:Case1} the foliation is algebraically integrable and the leaves are rational curves, which however do not exist in a torus.\smallskip

Let $T\coloneqq f^{-1}(t)$, $t\in f(D)$, be a generic fibre. The foliation ${\mathcal F}\subset{\mathcal T}_D$ induces a foliation ${\mathcal F}|_T\subset{\mathcal T}_T$
of the abelian variety $T$. It is well known that
the closure of a leaf of a foliation on an abelian variety is a translate of an abelian subvariety. Indeed, observing ${\mathcal O}_F\cong{\mathcal  F}|_T\subset {\mathcal T}_T\cong{\mathcal O}_F^{\oplus n}$ and writing $T={\mathbb C}^n/\Gamma$,  one finds that the leaves of the foliation ${\mathcal F}$ are given by the images under the natural projection ${\mathbb C}^n\twoheadrightarrow T$ of the translates of the line ${\mathbb C}\subset {\mathbb C}^n$ corresponding to ${\mathcal F}|_T\subset{\mathcal T}_T$.  The closure of the leaf through the origin then corresponds to the smallest linear subspace ${\mathbb C}^m\subset{\mathbb C}^n$ containing the given line and such that $\Gamma\cap {\mathbb C}^m\subset {\mathbb C}^m$ is a lattice.

Thus, if the abelian variety $T$ is known to be simple, which is 
frequently the case, then the assertion is immediate.
\smallskip

If $T$ is not simple, then  Abugaliev proceeds in two steps. The first
 is a result of general interest \cite[Thm.\ 0.5]{Ab1}.\footnote{The reader will observe that the result actually holds without assuming that $X$ is hyperk\"ahler or that $f$ is a Lagrangian fibration.}

\begin{lem}[Abugaliev]
Let $f\colon X\to B$ be a Lagrangian fibration of a projective hyperk\"ahler
manifold 
 and let $H\subset B$ be a very ample hypersurface
not contained in the discriminant divisor of $f$.

If $D=f^{-1}(H)\subset X$ is smooth, then for the generic fibre $T=f^{-1}(t)$, $t\in H$:
$${\rm Im}\left(\!\!\xymatrix{H^\ast(X,{\mathbb Q})\ar[r]^-{{\rm res}_{X,T}}& H^\ast(T,{\mathbb Q})}\!\!\right)={\rm Im}\left(\!\!\xymatrix{H^\ast(D,{\mathbb Q})\ar[r]^-{{\rm res}_{D,T}}& H^\ast(T,{\mathbb Q})}\!\!\right).$$
\end{lem}

Note that the left hand side is known to be isomorphic to $H^\ast({\mathbb P}^n,{\mathbb Q})$ according to results by Matsushita, Oguiso, Voisin, and Shen--Yin, see the survey \cite[Thm.\ 2.1]{HM} for references. In particular, it is of dimension one in each even degree.

\begin{proof}
The assertion is invariant under small deformations of $H$, which preserve the smoothness of $D$. 
One may assume that the intersection $H\cap\Delta$
with the discriminant locus is sufficiently generic such that $\pi_1(H\setminus\Delta)\twoheadrightarrow \pi_1(B\setminus\Delta)$ is surjective
(and in fact an isomorphism for $n>2$) by \cite[Lem.\ 1.4]{Deligne} applied to $B\setminus \Delta$.
In particular, the monodromy invariant parts of $H^\ast(T,{\mathbb Q})$ for the two families $X\to B$ and $D\to H$ coincide. Thus, Deligne's invariant cycle theorem
implies
$${\rm Im}({\rm res}_{X,T})=H^\ast(T,{\mathbb Q})^{\pi_1(B\setminus \Delta)}=
H^\ast(T,{\mathbb Q})^{\pi_1(H\setminus \Delta)}={\rm Im}({\rm res}_{D,T}),$$
which concludes the proof.
\end{proof}

The idea of the second step is that the family of tori obtained as closures of leaves
$L\subset \bar L\subset T$ contained in a fixed generic fibre  $T$ is distinguished and hence invariant under monodromy. This gives a cohomology class in $H^{2k}(T,{\mathbb Q})$ that is invariant under monodromy of the family $D\to H$. However, classes that are invariant under the monodromy of the family $X\to B$
are all powers of the polarization and, therefore, cannot be realized by proper subtori.

%



\subsection{}\label{sec:2ii-iii} (iv) $\Rightarrow$ (ii) $\Rightarrow$ (i): Assume first (iv) holds.
  By Corollary \ref{cor:Folvert}, the generic leaf $L$ is contained in a fibre of $D\to f(D)\subset B$. If we allow ourselves to use (iv) $\Rightarrow$ (i) in
Section \ref{sec:2iv-i}, then the closure $\bar L$ is the generic fibre which is a torus. The second implication (ii) $\Rightarrow$ (i) is
clear.

\subsection{} Let $T=f^{-1}(t)\subset X$ be a smooth fibre of a Lagrangian fibration $f\colon X\to B$. Then $T$ is isomorphic to an abelian variety and picking a point $x\in X$ allows one to write
$T$ as a torus $T=T_xT/\Gamma$. For a hypersurface $t\in H\subset B$, which we assume to be smooth at $t$, we let
 $D\coloneqq f^{-1}(H)$ be its pre-image. Since the symplectic structure  of $X$ provides an isomorphism $T_xT\cong T_t^\ast B$, the tangent space $T_tH\subset T_tB$, viewed as a line in $T_t^\ast B$, corresponds to a line  $\ell_H\subset T_xT$. The  image of this line under $T_xT\to T$ gives the leaf through $x\in D$ of the characteristic foliation on $D$. If $H$ is chosen very general, then the line 
 $\ell_H\subset T_xT$ is very general and, therefore, its image in the quotient
 $T=T_xT/\Gamma$ is dense. This proves the following.

\begin{prop} For a fixed smooth fibre $T=f^{-1}(t)\subset X$ of a Lagrangian fibration $f\colon X\to B$ and a very general smooth hypersurface $t\in H\subset B$ the leaf of the characteristic foliation
of $D=f^{-1}(H)$ through a point $x\in T$ is dense in the fibre $T$.\qed
\end{prop}

\section{Case (3): Dense leaves}\label{sec:Case3}
Again, the equivalence of the conditions (i)-(iv) holds in dimension four, but assumes that the abundance conjecture holds for $D$.  Hwang and Viehweg \cite{HV} showed that if $D$ is of general type, the foliation is not algebraically integrable.  In the converse direction, in dimension four, Amerik and Campana \cite{AC} proved that the foliation is algebraically integrable if and only if $D$ is uniruled. The assertion that $D$ being nef and big implies density
of the leaves is due to Abugaliev and  the main result of \cite{Ab2}.

\subsection{}\label{sec:3i-iii} (i) $\Rightarrow$ (iii):  
We assume that $\dim\bar L=2n-1$ and want to exclude that $q(D)<0$ or $q(D)=0$.

First, by the results of Section \ref{sec:Case1}, we know that the three conditions $q(D)<0$, $D$ uniruled, and $\dim \bar L=1$ are all equivalent. Hence, $q(D)<0$ is excluded for $\dim\bar L>1$. 

Next suppose $q(D)=0$. Then, by Proposition \ref{prop:Abtrick}, $D$ is nef and hence also
$\omega_D\cong{\mathcal O}(D)|_D$ is. Assuming the abundance conjecture for $D$, we conclude that $\omega_D$ and, therefore, $\ko(D)$ are semi-ample, cf.\ the argument in Section \ref{sec:2iii-iv}. Hence, $X$ comes with a Lagrangian fibration $X\to B$ such that some multiple of $D$ is the pre-image of a divisor in $B$. However, then the generic leaf is not dense, as a leaf passing through a smooth fibre stays in this fibre, cf.\ the discussion in Section \ref{sec:2iv-i}. 

\subsection{}\label{sec:3iii-iv} (iii) $\Rightarrow$ (iv):  We assume $q(D)>0$.
Clearly, if $D$ is ample, then by adjunction $\omega_D\cong \ko(D)|_D$ is ample as well and, therefore, $D$ is of general type. If $D$ is only nef, then a priori also
$\omega_D$ is only nef. However, $q(D)>0$ implies $\int_D{\rm c}_1(\omega_D)^{2n-1}>0$, i.e.\ $\omega_D$ is big and nef. By the Kawamata--Viehweg vanishing theorem $H^i(D,\omega_D^k)=0$ for $k>1$ and $i>0$ and by the Hirzebruch--Riemann--Roch formula $h^0(D,\omega_D^k)\sim k^{2n-1}$. Hence, $D$ is of general type. Since by Proposition \ref{prop:Abtrick}, any smooth hypersurface $D\subset X$ with $q(D)\geq0$ is nef, this concludes the proof.
\smallskip

Here is an alternative argument not relying on Proposition \ref{prop:Abtrick}:
Since $D$ is contained in the interior of the positive cone, it is also contained
in the interior of the pseudo-effective cone and, therefore, big by \cite[Prop.\ 2.2.6]{Laz}, i.e.\ $h^0(X,\ko(kD))\sim k^{2n}$. Using $\omega_D\cong\ko(D)|_D$ eventually shows that $\omega_D$
is big, i.e.\ that $D$ is of general type.

\subsection{}\label{sec:3iv-iii} (iv) $\Rightarrow$ (iii): We assume that $D$ is of general type and want to prove $q(D)>0$ by excluding the other two  possibilities $q(D)<0$ and $q(D)=0$.

Suppose $q(D)<0$. Then by virtue of Section \ref{sec:1iii-iv}, e.g.\ by applying the Miyaoka--Mori numerical criterion for uniruledness \cite{MM, Miyaoka}, we know that $D$ is uniruled, so in particular not of general type.

Next, suppose that $q(D)=0$, which implies $\int_D{\rm c}_1(\omega_D)^{2n-1}=(D)^{2n}=0$. Now use again Proposition \ref{prop:Abtrick} to conclude that $\ko(D)$ and hence $\omega_D\cong\ko(D)|_D$ are nef. However, a nef divisor $E$ on a projective variety $Z$ of dimension $m$  is big, i.e.\ $h^0(Z,\ko(kE))\sim k^m$, if and only if $(E)^m>0$, see \cite[Thm.\ 2.2.16]{Laz}. Since $D$ is assumed to be of general type and so $h^0(D,\ko(kD)|_D)\sim k^{2n-1}$, this 
is a contradiction.

\subsection{}\label{sec:3iii-i} (iii) $\Rightarrow$ (i): We assume $q(D)>0$ and want to show that the leaves are dense.  We have seen already that $q(D)>0$ implies that $D$ is big and nef.  The first step is to prove
a version of the Lefschetz hyperplane theorem \cite[Prop.\ 3.1]{Ab2}.

\begin{lem}\label{lem:hardLefschetz}
Assume $D\subset X$ is a smooth hypersurface of a hyperk\"ahler manifold that is   big and nef. Then the restriction 
$$H^i(X,\QQ)\congpf H^i(D,\QQ)$$
is an isomorphism for $i<\dim (D)=2n-1$.
\end{lem}

\begin{proof} For an ample hypersurface this is the content of the Lefschetz hyperplane theorem \cite[Thm.\ 3.1.17]{Laz}.
If $D$  is just big and nef,  Kawamata--Viehweg vanishing still shows that all higher cohomology groups $H^i(X,\ko(kD))$, $i>0$, are trivial. Hence, $D$ deforms sideways with $X$ in any family ${\mathcal X}\to \Delta$ for which the line bundle $\ko(D)$ deforms. However, the very general fibre
${\mathcal X}_t$ of the universal such deformation has Picard number one. Therefore, a generic deformation $D_t\subset {\mathcal X}_t$ of $D\subset X$ is ample \cite[Thm.\ 3.11]{Huy}. Hence, $H^i({\mathcal X}_t,\QQ)\congpf H^i(D_t,\QQ)$ for $i<\dim (D_t)$ by the classical  Lefschetz hyperplane theorem. Since the assertion is topological, this suffices to conclude.\footnote{The original proof in \cite{Ab2} uses the Kodaira--Akizuki--Nakano vanishing theorem. The above argument is quicker, but uses deformation theory and the projectivity criterion for hyperk\"ahler manifolds.}
\end{proof}

The key step is the following result \cite[Prop.\ 4.1]{Ab2}.

\begin{prop}[Abugaliev]\label{prop:Abu2}
A smooth hypersurface $D\subset X$ which is big and nef (or,  equivalently, a smooth hypersurface with $q(D)>0$, cf.\ Proposition \ref{prop:Abtrick}) cannot be covered by coisotropic varieties of codimension two in $X$.
\end{prop}

\begin{proof}
Recall that a subvariety $Z\subset X$ of codimension two is called coisotropic if
the kernel of $\sigma|_Z\colon \kt_Z\to \Omega_Z$ (over the smooth locus) is a sheaf of rank two, see Section \ref{sec:fol}.

First observe that by Lemma \ref{lem:hardLefschetz}, that for any subvariety $Z\subset D$ there exists a class
$\alpha\in H^2(X,\QQ)$ with $\alpha|_D=[Z]\in H^2(D,\QQ)$. Clearly, the class $\alpha$ is of type $(1,1)$. On the other hand, if $Z\subset X$ is a coisotropic
subvariety of codimension two, then $0=[Z]\wedge\sigma^{n-1}\in H^{2n+2}(X,\CC)$. So, if $Z\subset D$ is coisotropic
and we write $[Z]=\alpha|_D\in H^2(D,\QQ)$, then  $0=[D]\wedge\alpha\wedge\sigma^{n-1}\in H^{2n+2}(X,\CC)$, which
implies $\int_X[D]\wedge\alpha\wedge\sigma^{n-1}\wedge\bar\sigma^{n-1}=0$ and, therefore, $q(D,\alpha)=0$. 
Now use the well known formula
\begin{equation}\label{eqn:gamma}q(\gamma_1,\gamma_2)\cdot\int_X\gamma_1^{2n}=2q(\gamma_1)\cdot\int_X\gamma_1^{2n-1}\wedge\gamma_2,
\end{equation} cf.\ \cite[Exer.\ 23.2]{HuyGHJ} to deduce from $q(D)>0$ 
  and $q(D,\alpha)=0$, that $0=\int_X[D]^{2n-1}\wedge\alpha=\int_Z[D]|_Z^{2n-2}$.
   If $D$ is ample, this is absurd. So we proved the stronger statement that an ample, smooth
hypersurface $D\subset X$ does not contain any coisotropic subvariety of codimension two.\footnote{In dimension four this says that a smooth ample
hypersurface does not contain any smooth Lagrangian surface, see Section \ref{sec:exasing}.}
If $D$ is only big and nef, then $q(D,\alpha)=0$ still implies $q(\alpha)<0$ by the
Hodge index theorem, which in turn, by a formula
\cite[Lem.\ 4.2]{Ab2} similar to (\ref{eqn:gamma}), gives
$\int_X[D]^{2n-2}\wedge\alpha\wedge\alpha<0$. However, if $D$ can be covered by
coisotropic varieties of codimension two, then there exist two such $Z_1,Z_2\subset D$
realising the same class  $\alpha|_D=[Z_1]=[Z_2]$ and then  $Y\coloneqq Z_1\cap Z_2\subset D$ is empty or of codimension two in $D$, which leads to the contradiction  $0\leq\int_YD|_Y^{2n-3}=\int_X[D]^{2n-2}\wedge\alpha\wedge\alpha<0$.
 \end{proof}

Recall that a subvariety $Z\subset D$ is called \emph{invariant} under the 
characteristic foliation of the smooth hypersurface $D$ if the leaf through any
$x\in Z$ is contained in $Z$ or, equivalently, if ${\mathcal F}|_Z\subset\kt_D|_Z$ is contained in $\kt_Z\subset\kt_D|_Z$ (over the smooth locus of $Z$).

The following result \cite[Thm.\ 0.5]{Ab2} is now a consequence of the above discussion. It concludes the proof of (iii) $\Rightarrow$ (i) in Case (3).

\begin{thm}[Abugaliev]\label{lem:Abu2}
Assume $D\subset X$ is a smooth hypersurface of a hyperk\"ahler manifold
$X$ satisfying $q(D)>0$. Then the generic leaf of the characteristic foliation on $D$ is Zariski dense.
\end{thm}

\begin{proof}
If the generic leaf $L\subset D$ is not Zariski dense, then its Zariski closure $\bar L\subset D$ defines a proper closed subvariety $Z\subset D$. The family of all such leaves gives a covering family $\{Z_t\}$ of $D$. Assume first that
$Z_t\subset D$ is of codimension two in $X$. Since $Z_t=\bar L$ is clearly
invariant under the characteristic foliation and hence coisotropic by Corollary \ref{cor:Folvert}, (ii), this is a contradiction to Proposition \ref{prop:Abu2}.

If the subvarieties $Z_t$ are of higher codimension, taking unions produces a covering family $\{Z'_s\}$ of $D$ consisting of subvarieties of codimension two in $X$ and such that each $Z'_s$ is a union of $Z_t=\bar L$. In particular, again by Corollary \ref{cor:Folvert}, (ii),
 each $Z'_s$ is coisotropic and one can  conclude as before.
\end{proof}
\subsection{}\label{sec:3iv-i} (iv) $\Rightarrow$ (i): Of course, this direction is a consequence of the implications proved before, but we wish to mention
a weaker statement due to Hwang--Viehweg that motivated much of the later work on characteristic foliations. They proved \cite[Thm.\ 1.2]{HV}
that the characteristic foliation of a smooth hypersurface $D\subset X$ cannot be algebraic or, in other words, that $\dim \bar L>1$.
\section{Alternative summary}
We think it is instructive to present the discussion concerning the equivalence of the two conditions (iii) and (iv) in a somewhat differently structured way, making it more evident where and how foliations are used. 

\subsection{} (iii) $\Rightarrow$ (iv): For a smooth hypersurface
$D\subset X$ one wants to show that the sign of $q(D)$ largely determines the geometry of $D$.

 This part only involves more or less classical results and Proposition \ref{prop:Abtrick}, i.e.\ the nefness of $D$ if $q(D)\geq0$. Recall that the proof of Proposition 
\ref{prop:Abtrick} used foliations in an essential way.
\medskip

$\bullet$ Assume $q(D)<0$. Then $\int_D{\rm c}_1(\omega_D)\cdot H^{2n-2}<0$ and by \cite{MM,Miyaoka} $D$ is uniruled.
\smallskip

$\bullet$ Assume $q(D)=0$. Then $D$ and hence $\omega_D$ are nef by Proposition \ref{prop:Abtrick}. Using abundance conjecture for $D$ combined
with \cite{Dem}, see footnote on page \pageref{foot:DHP}, one finds that $D$ is semi-ample. Therefore, ${\mathcal O}(kD)$
defines a Lagrangian fibration $f\colon X\to B$ for some $k>0$ and hence $D=f^{-1}(f(D))$.
\smallskip

$\bullet$ Assume $q(D)>0$. In this case $D$ is of general type for which we presented two proofs: The one not using Proposition \ref{prop:Abtrick} just observed that under
these assumptions $D$ is in the interior of the pseudo-effective cone and hence big.

\subsection{} (iv) $\Rightarrow$ (iii): The geometry of $D$ determines the sign of $q(D)$. Again, only Proposition \ref{prop:Abtrick} is used.

\medskip

$\bullet$ Assume $D$ is uniruled. Then $q(D)<0$ is proved by excluding $q(D)>0$
and $q(D)=0$. If $q(D)>0$, then $D$ is of general type as explained above.
To exclude $q(D)=0$, one distinguishes two cases: First, if
$D$ is in the
boundary of the movable cone, then $\omega_D={\mathcal O}(D)|_D$ is a limit
of effective divisors and, therefore, pseudo-effective which contradicts the assumption that $D$ is uniruled.
Second, if $D$ is not contained in the boundary of the movable cone, then $D$
is in the interior of the pseudo-effective cone. Hence, $D$ and $\omega_D$ are big, contradicting again the assumption on $D$.
Alternatively, one could apply Proposition \ref{prop:Abtrick} to see
that $D$ and hence $\omega_D$ are nef, but the latter clearly contradicts $D$ being uniruled.
\smallskip

$\bullet$ Assume $D=f^{-1}(H)$ is the set theoretic pre-image of a hypersurface $H\subset B$ in the base of a Lagrangian fibration $f\colon X\to B$.
Then the classes $[D], f^\ast[H]\in H^2(X,{\mathbb Z})$ are proportional.
Since $[H]^{n+1}=0$, also $[D]^{n+1}=0$ in $H^{2n+2}(X,{\mathbb Z})$ and, therefore, $q(D)=0$


\smallskip

$\bullet$ Assume $D$ is of general type. Then $q(D)>0$ is proved by excluding $q(D)<0$ and $q(D)=0$. Indeed, the former would imply that $D$ is uniruled as explained before. 
The latter is excluded by observing that $\omega_D$ is nef by Proposition 
\ref{prop:Abtrick} and big, for $D$ is of general type. However, this implies
$\int_D{\rm c}_1(\omega_D)^{2n-1}>0$ which excludes $q(D)=0$.

\section{Examples}

\subsection{}\label{sec:exasmooth}

We provide examples of divisors for the first two situations in the case that $X$ is the Hilbert scheme $X=S^{[2]}$ of a K3 surface $S$.\smallskip

(i) The natural example for Case (1) is  the exceptional divisor $D=E$ of the Hilbert--Chow morphism $\pi\colon S^{[2]}\rightarrow S^{(2)}$. It is well known that $q(E)=-2$ and it is a $\mathbb{P}^1$ bundle over the diagonal  $S\subset S^{(2)}$. More explicitly:\smallskip

$\bullet$ The divisor $E$ is naturally isomorphic to $\mathbb{P}(\Omega_S)$.\smallskip

$\bullet$ The restriction of the symplectic form of $S^{[2]}$ to $E$ is the pullback of the symplectic form on $S$ via the projection $\mathbb{P}(\Omega_S)\rightarrow S$.\smallskip

$\bullet$ The characteristic foliation ${\mathcal F}$ is the relative tangent bundle ${\mathcal T}_\pi$ of the map $\pi\colon\mathbb{P}(\Omega_S)\rightarrow S$.\smallskip

$\bullet$ The leaves are the $\mathbb{P}^1$ contracted by $\pi$,
which via the identification with $\mathbb{P}(\Omega_S)$ is just the projection to $S$. Hence, the diagonal $S\subset X^{(2)}$  is the space of leaves $D/{\mathcal F}$.
\medskip

(ii) Assume that $S$ admits a genus one fibration. Then 
$S^{[2]}$ comes with a natural Lagrangian fibration $\pi\colon S^{[2]}\to{\mathbb P}^2\cong({\mathbb P}^1)^{[2]}$ over the Hilbert scheme of two points on $\mathbb{P}^1$.
The pre-image  $D$ of a generic line $\ell\subset{\mathbb P}^2$ is a hypersurface
which is smooth by Bertini and satisfies $q(D)=0$. The leaves of the characteristic foliation on $D$ are contained in the fibres of $D\twoheadrightarrow \ell$ but the very general ones are not compact, i.e.\ they are dense in the fibres. This follows from the implications {\rm (iii)} $\Rightarrow$ {\rm (iv)} $\Rightarrow$ {\rm (i)}
proved in Sections \ref{sec:2iii-iv} and \ref{sec:2iv-i}.

The situation changes if $\ell\subset{\mathbb P}^2$ is special. For example, if 
$S_{0}$ is a smooth fibre of $S\rightarrow \mathbb{P}^1$, then the pre-image of $\ell_0\coloneqq\{\{0,t\}\mid t\in{\mathbb P}^1\}\subset({\mathbb P}^1)^{[2]}$
is the hypersurface $ D_0\coloneqq \pi^{-1}(\ell_0)=\{\{p,q\}\mid p \in S_{0}\}$. It still  satisfies
$q(D_0)=0$, but it is not smooth. In fact, it is not even normal, its normalisation
is the natural map ${\rm Bl}_{\Delta}(S_0\times S)\to D_0$ from the blow-up 
in the diagonal in $S_0\times S_0$ which restricts to the degree two map $S_0^2\to S_0^{[2]}\subset D_0\subset S^{[2]}$ and is injective on the complement.
The fibre of $\pi\colon D_0\to \ell_0$ over a point $\{0,t\}$, $t\ne0$, is the 
surface $S_0\times S_t$, which is smooth for all but finitely many $t$.
The singularities of $D_0$ have an effect on the characteristic foliation (of the smooth part) of $D_0$: The leaves in the generic fibre $\pi^{-1}(\{0,t\})$ are the curves $S_0\times x$, $x\in S_t$, which in particular are not dense in the smooth(!) fibre $S_0\times S_t$.\footnote{We wish to thank the referee for this observation.}

%
%

\subsection{}\label{sec:exasing}
As we have just seen, if the hypersurface $D\subset X$ is not smooth,  then typically the conditions
(i)-(iv) are not equivalent.\smallskip

(i) Let us first discuss this in Case (2). Consider the pre-image
$f^{-1}(\Delta)\subset X$  of the discriminant divisor of a Lagrangian fibration $f\colon X \rightarrow B$. Note that even for $\Delta$ irreducible its pre-image may be reducible. By \cite[Thm.\ 1.2]{HO} the characteristic foliation of any irreducible component of $f^{-1}(\Delta)$ is algebraically integrable. Assume that there is a component of $D$ of $f^{-1}(\Delta)$ such that $D=f^{-1}(f(D))$. This happens for instance when  $X$ is general among the hyperk\"ahler manifolds with a Lagrangian fibration. Then, $q(D)=0$ but its characteristic foliation is algebraically integrable. This divisor satisfies (iii) and (iv) of Case (2) but not (i). 


\medskip

(ii) We turn to Case (3).
Consider a smooth cubic fourfold $Y\subset {\mathbb P}^5$ and its Fano
variety of lines $X\coloneqq F(Y)$, which is a hyperk\"ahler fourfold.
The set of lines contained in a hyperplane section $Y\cap H$  is
a Lagrangian surface $F(Y\cap H)\subset X$ which for generic $H$ is 
smooth and of general type. For a one-dimensional family $\{Y\cap H_t\}$ of hyperplane sections these Lagrangian surfaces sweep out a hypersurface $D\subset X$. Then $q(D)>0$, since for a generic cubic fourfold the Picard number of $X=F(Y)$  is one.

According to Corollary \ref{cor:Folvert}, (i), any leaf that intersects a generic $F(Y\cap H_t)$ is contained in it. However, if $D$ were smooth, then the results if Section \ref{sec:3iii-i} would imply that the generic leaf is dense. Contradiction. Hence, 
for no one-dimensional family of hyperplane section $\{Y\cap H_t\}$
can the associated hypersurface $D$  be smooth.

In particular this is an example of a singular divisor that satisfies (iii) and (iv) of Case (3) but not (i).

More abstractly, a smooth ample hypersurface $D\subset X$ in a hyperk\"ahler fourfold does not contain any Lagrangian surface, cf.\  the proof of Proposition \ref{prop:Abu2}. In particular, for a  general cubic fourfold $Y$, a smooth Lagrangian surface $F(Y\cap H_t)$ cannot be contained in any smooth divisor of $X=F(Y)$.


\end{document}